\documentclass[12pt,a4paper]{article}

\usepackage{amssymb}
\usepackage{amsmath, amsthm}
\usepackage{arydshln}
\usepackage{epsfig}
\usepackage{setspace}

\usepackage{geometry}

\numberwithin{equation}{section}

\newcommand{\mConv}{\mathop{\rm mConv} }
\newcommand{\dist}{\mathop{\rm dist} }

\newtheorem{Thm}{Theorem}[section]

\newtheorem{Lem}[Thm]{Lemma}
\newtheorem{Cor}[Thm]{Corollary}
\theoremstyle{definition}
\newtheorem*{Clm}{Claim}

\title{Two flags in a semimodular lattice \\
generate an antimatroid}
\author{Koyo Hayashi\footnote{Department of Information and System Engineering, 
		Chuo University, Tokyo 112-8551, Japan.		
		\texttt{hayashi@ise.chuo-u.ac.jp}} \quad \quad
Hiroshi Hirai\footnote{
Department of Mathematical Informatics, 
Graduate School of Information Science and Technology,  
The University of Tokyo, Tokyo, 113-8656, Japan.
\texttt{hirai@mist.i.u-tokyo.ac.jp}}
}

\begin{document}
	\maketitle
	\begin{abstract}
		A basic property in a modular lattice is 
		that any two flags generate a distributive sublattice.
		It is shown (Abels 1991, Herscovic 1998)
		that two flags in a semimodular lattice no longer generate such a good sublattice, whereas shortest galleries connecting them
		form a relatively good join-sublattice.
		In this note,  
		we sharpen this investigation to establish an analogue of
		the two-flag generation theorem for a semimodular lattice.
		We consider the notion of 
		a modular convex subset,  which is a subset 
		closed under the join and meet only for modular pairs, 
		and show that the modular convex hull 
		of two flags in a semimodular lattice of rank $n$ is isomorphic 
		to a union-closed family on $[n]$. 
		This family uniquely determines an antimatroid,
		which coincides with the join-sublattice of shortest galleries of the two flags.
\end{abstract}
\section{Introduction}

One of the basic and classical results in lattice theory 
is that 
any two maximal chains ({\em flags}) in a modular lattice generates 
a distributive sublattice:
\begin{Thm}[\cite{Birkhoff}]\label{thm:modular}
	Let $L$ be a modular lattice with a finite rank, and let $C, D$ be flags of $L$.
	Then the sublattice generated by $C \cup D$
	is a distributive lattice.
\end{Thm}

The present note addresses an extension of this theorem 
for a semimodular lattice. 
The distributive sublattice in the above theorem
can be interpreted in terms of the graph of all flags, 
which is obtained by joining two flags differing on exactly one element. 
A path in this graph is usually referred to as a {\em gallery}. 
Abels~\cite{AbelsA,AbelsB}
studied the gallery-distance of 
two flags in a semimodular lattice, and showed that 
for the case of a modular lattice, 
the union of all shortest galleries between two flags coincides with the distributive sublattice generated by them. 
He also showed that in a semimodular lattice, 
the elements in shortest galleries between two flags
form a join-closed subset of the original lattice. 
This direction was further addressed by Herscovic~\cite{Herscovici98} in terms of reduced decompositions of the Jordan-H\"older permutation.
He showed that the sublattice generated by two flags in a semimodular lattice no longer has good properties; 
it needs not to be ranked and is not known to be finite.  
He claims that in a semimodular lattice
the join-closed subset of shortest galleries of two flags
is a natural correspondent for  
the sublattice of two flags 
in a modular lattice.

In this note, we sharpen these investigations to establish an analogue of Theorem~\ref{thm:modular} for a semimodular lattice.
The above ``generation by 
two flags" 
is meant as repeated additions 
of join and meet of pairs from the elements in the flags.
We here consider a restricted generation by join and meet only for {\em modular pairs}---pairs of elements satisfying the modularity equality for their rank.
We introduce a {\em modular convex subset} 
as a subset that is closed under join and meet for modular pairs. 
The main result in this note is 
an extension of Theorem~\ref{thm:modular}
in terms of the modular convex hull, 
which clarifies the above join-closed subset of two flags 
and establishes a link to antimatroids.
Specifically, we show that in a semimodular lattice with rank $n$ the modular convex hull 
of two flags is isomorphic, as a poset, to a union-closed family ${\cal A} \subseteq 2^{[n]}$ having a (maximal) chain of length $n$. 
Such a union-closed family is referred here to as a {\em pre-antimatroid}, 
since it uniquely determines an {\em antimatroid}~\cite{Edelman,EdelmanJamison,Greedoid} as the union of all maximal chains.
This construction of an antimatroid 
from a union-closed family was 
implicit in \cite{Greedoid}.
Its importance was recently recognized by~\cite{Doignon,Yoshikawa} 
in the literature of {\em Knowledge Space Theory (KST)}~\cite{LearningSpace}.
We show that this antimatroid coincides with the above join-closed subset of shortest galleries between the two flags.

Various classes of antimatroids and semimodular lattices are known. 
It would be an interesting future research  
to study how they are related via the modular convex hull of two flags.
Our proof is constructive, and will be useful for such study.

\section{Preliminaries}
A {\em lattice} is a partially ordered set $L$ 
in which every pair $p,q$ of elements has join 
$p \vee q$ (the minimum common upper bound) 
and meet $p \wedge q$ (the maximum common lower bound). 
The partial order is denoted by $\leq$. 
By $p < q$ we mean $p \leq q$ and $p \neq q$. 
For a pair $p,q$ with $p \leq q$, the interval $[p,q]$ is defined as the set of elements $z$ with $p \leq z \leq q$.
We say that $q$ covers $p$ if $p \neq q$ and $[p,q]=\{p,q\}$.
A totally ordered subset is called a {\em chain}, which will be denoted simply as $p_1 < p_2 < \cdots$.  
The length of a chain is its cardinality minus one.
A subset $S$ of a lattice is said to be {\em join-closed} if $p,q \in S \Rightarrow p \vee q \in S$. (Such a subset is called a {\em subsemilattice} in \cite{AbelsB} and {\em join-sublattice} in \cite{Herscovici98}.)
Two posets $P,Q$ are said to be {\em isomorphic} 
if there is an order-preserving bijection from $P$ to $Q$.
An order-preserving injection from $P$ to $Q$ is called an {\em embedding}.  

A lattice $L$ considered in this paper 
has the maximum element and minimum element, 
which are denoted by $1$ and $0$. 
The {\em rank} of an element $p$ is defined 
as the maximum length of a chain from $0$ to $p$, 
and is denoted by $r(p)$.
The rank of the lattice $L$ is defined as the rank of $1$.
We only consider lattices with a finite rank.

For an positive integer $n$, let $[n] := \{1,2,\ldots,n\}$.
The family $2^{[n]}$ of all subsets of $[n]$ is viewed as 
a (Boolean) lattice by the inclusion order.
A subfamily ${\cal F} \subseteq 2^{[n]}$ 
is regarded as a poset with respect to this order.

An (upper-){\em semimodular} lattice is 
a lattice $L$ satisfying:
\begin{itemize}
	\item	if $a$ covers $a \wedge b$, then $a \vee b$ covers $b$.
\end{itemize}
It is known \cite[Theorem 375]{Gratzer} 
that a semimodular lattice is precisely 
a lattice whose rank function $r$ satisfies the semimodularity inequality: 
\begin{equation}\label{eqn:semimodular}
r(p) + r(q) \geq r(p \wedge q) + r(p \vee q) \quad (p,q \in L).
\end{equation}
A pair $(p,q)$ of elements is said to be {\em modular} 
if it holds $(x \vee p) \wedge q = x \vee (p \wedge q)$
for every element $x \in L$ with $x \leq q$. 
In a semimodular lattice, modular pairs $(p,q)$ are precisely those pairs which satisfies (\ref{eqn:semimodular}) in equality; see \cite[Theorem 381]{Gratzer}.
In particular, a pair $(p,q)$ is modular if and only if $(q,p)$ is modular (which is called M-symmetry).

An {\em antimatroid} on a finite set $E$ 
is a family ${\cal A} \subseteq 2^E$ of subsets
satisfying the following properties:
\begin{itemize}
	\item[(A1)] $\emptyset, E \in {\cal A}$.
	\item[(A2)] For $X, Y \in {\cal A}$, it holds $X \cup Y \in {\cal A}$.  
	\item[(A3)] For a nonempty set $X \in {\cal A}$, there is $e \in X$ 
	such that $X \setminus \{e\} \in {\cal A}$. 
\end{itemize}
See e.g., \cite{ConvexGeometry,Greedoid} for antimatroids.
An antimatroid is a semimodular lattice with respect to inclusion order $\subseteq$. 
Indeed, $X \vee Y$ equals $X \cup Y$, 
and $X \wedge Y$ equals the union of all members contained in $X \cap Y$.
The rank is given by the cardinality $X \mapsto |X|$, 
from which the semimodularity of the rank follows.
It is known~\cite{Edelman} that an antimatroid is precisely a realization of
a special semimodular lattice, called a {\em join-distributive lattice}, in $2^{E}$.

We introduce a weaker notion of an antimatroid
that determines an antimatroid uniquely. 
A {\em pre-antimatroid} on a finite set  $E$ 
is a family ${\cal K} \subseteq 2^E$ of subsets satisfying (A1), (A2), and 
\begin{itemize}
	\item[(A3$'$)]  There is a (maximal) chain of length $|E|$.
\end{itemize}
For a pre-antimatroid ${\cal K}$, let ${\cal K}^* \subseteq {\cal K}$
be defined as the union of all chains of length $|E|$.
\begin{Lem}[{\cite{Doignon}; see \cite{Yoshikawa}}]
	Let ${\cal K}$ be a pre-antimatroid. 
	Then ${\cal K}^*$ is an antimatroid.
\end{Lem}
\begin{proof}
	It suffices to verify that ${\cal K}^*$ satisfies the union-closedness (A2).
	Indeed, 
	if ${\cal X}, {\cal Y}$ are chains of length $|E|$, 
	$X \in {\cal X}$, and $Y \in {\cal Y}$,
	then $\{X' \cup Y' \mid  X' \in {\cal X},Y' \in {\cal Y}\} \subseteq {\cal K}$ contains a chain of length $|E|$ having $X \cup Y$.
\end{proof}

Consider the family of all flags in a semimodular lattice of rank $n$, 
and endow it with a graph structure as: Two flags $C,D$ are adjacent 
if and only if $C \cap D = n-1$. 
A path in this graph is called a {\em gallery}.
The {\em flag-distance} $\dist (C,D)$ of two flags $C,D$ is 
the minimum length of a shortest gallery 
between $C$ and $D$.
It is known \cite{AbelsA} that the flag-distance
can be computed from Jordan-H\"older permutation;
see Theorem~\ref{thm:inversion} below.
\section{Result}\label{sec:result}

Let $L$ be a semimodular lattice with rank $n$. 
A subset $S \subseteq L$ is said to be {\em modular convex} if for every modular pair of elements $p,q \in S$ 
it holds $p \vee q, p \wedge q \in S$.
The {\em modular convex hull} of $S$, denoted by $\mConv S$, is defined as the minimal modular convex set containing $S$.
Restricting the partial order of $L$, 
the modular convex hull $\mConv S$ is 
regarded as a subposet of $L$. 
It is clear that in a modular lattice $\mConv S$ is equal to 
the sublattice generated by $S$.

The main result of this note clarifies the relationship among
the modular convex hull 
of two flags, their flag-distance interval,  
and an antimatroidal structure.
\begin{Thm}\label{thm:main}
	Let $C, D$ be flags of $L$.
	\begin{itemize}
		\item[\rm (1)]  $\mConv (C \cup D)$ 
		is isomorphic to a pre-antimatroid ${\cal K}$ on $[n]$.
		\item[\rm (2)] A flag $F$  belongs to 
		$\mConv( C \cup D)$
		if and only if $F$ belongs to a shortest gallery between $C$ and $D$. In particular, 
		the union of all flags of all shortest galleries between $C$ and $D$
		is isomorphic to the antimatroid ${\cal K}^*$.
	\end{itemize}
\end{Thm}

We start the proof.
Suppose that $C$ and $D$ are given by $0 = c_0 < c_1 < \cdots < c_n = 1$
and  
$0 = d_0 < d_1 < \cdots < d_n = 1$, respectively.
Consider $L' := \{ p \in L \mid p \geq d_1\}$, 
which is a semimodular sublattice with rank $n-1$.
Define flags $C',D'$ of $L'$ by
$C' := C \vee d_1 = \{c_i \vee d_1 \}_{i=0,1,\ldots,n}$ and 
$D' := D \vee d_1 = \{d_i \}_{i=1,\ldots,n}$.
By semimodularity, $C'$ is actually 
a flag of $L'$.
Let $k$ be the maximum index for which $c_k \not \geq d_1$.
Then $C'$ is written as 
\[
d_1 = c_0 \vee d_1 < c_1 \vee d_1 < \cdots <
c_k \vee d_1 = c_{k+1} < \cdots < c_n = 1.
\]
Define $\Delta \subseteq [0, c_k]$ by
\begin{equation}\label{eqn:Delta}
\Delta := \{ c_k \wedge q \mid q \in [d_1, c_{k+1}] \cap \mConv (C' \cup D'): \mbox{$(c_k,q)$ is a modular pair}\}.
\end{equation}
\begin{Lem}\label{lem:1}
	$\mConv (C \cup D)$ equals $\mConv (C' \cup  D') \cup \Delta$, and is join-closed.
\end{Lem}
\begin{proof}
	Since $(d_1, c_i)$ is a modular pair,
	it holds $C' \subseteq \mConv (C \cup D)$ and $\mConv ({C}' \cup {D}') \subseteq \mConv ({C} \cup {D})$. 
	By definition (\ref{eqn:Delta}), it holds $\mConv ({C} \cup {D}) \supseteq \mConv ({C}' \cup {D}') \cup \Delta$.
	Also $(c_k, c_i \vee d_1)$ $(0 \leq i \leq k)$ is a modular pair 
	with $c_k \wedge (c_i \vee d_1) = c_i$.
	This implies that $\mConv ({C}' \cup {D}') \cup \Delta$ contains both ${C}$ and ${D}$.
	Therefore we have
	\begin{equation}\label{eqn:inclusion}
	\mConv ({C} \cup {D}) \supseteq \mConv ({C}' \cup {D}') \cup \Delta \supseteq {C} \cup {D}.
	\end{equation}
	We show:
	\begin{Clm}
		If $\mConv (C' \cup D')$ is join-closed and modular convex, then so is 
		$\mConv (C' \cup  D') \cup \Delta$. 
    \end{Clm}
	Then, by (\ref{eqn:inclusion}) and induction, 
	we have the statement.
	Take any $u,v \in \mConv ({C}' \cup {D}') \cup \Delta$. 
	Our goal is to show $u \vee v \in \mConv ({C}' \cup {D}') \cup \Delta$, 
	and $u \wedge v \in \mConv ({C}' \cup {D}') \cup \Delta$ if $(u,v)$ is modular.
	
	Case 1: $u \in \mConv ({C}' \cup {D}'), v \in \Delta$.
	Write $v$ by $v = q \wedge c_k$ as in (\ref{eqn:Delta}).
	Since $q$ covers $v$ (by modularity), it must hold $q = d_1 \vee v$.
	Then $u \vee v =  u \vee d_1 \vee v = u \vee q$.
	Therefore, if $\mConv ({C}' \cup {D}')$ is join-closed, then $u \vee v \in \mConv ({C}' \cup {D}')$.
	Suppose that $(u,v)$ is a modular pair.
	Then $(u,q)$ is also a modular pair.
	Indeed, $u \wedge q$ covers $u \wedge v$ or equals $u \wedge v$; the latter is impossible by $u \wedge q \geq d_1 \not \leq u \wedge v$.
	Therefore $u \vee q  = u \vee v \in \mConv ({C}' \cup {D}')$,  
	and $u \wedge q \in [d_1, c_{k+1}] \cap \mConv ({C}' \cup {D}')$.
	Also $(u \wedge q, c_k)$ is 
	a modular pair with $u \wedge q \wedge c_k = u \wedge v$.
	Thus $u \wedge v \in \Delta$.
	
	Case 2: $u,v \in \Delta$.
	Note that $z \mapsto d_1 \vee z$ 
	is an embedding from $\Delta$ to $[d_1, c_{k+1}] \cap \mConv ({C}' \cup {D}')$ 
	such that $c_k \wedge (d_1 \vee z) = z$; see \cite[Lemma 384]{Gratzer}.
	Note that $u \vee d_1, v \vee d_1 \in [d_1, c_{k+1}] \cap \mConv ({C}' \cup {D}')$ 
	with $u = (u \vee d_1) \wedge c_k$ and $v = (v \vee d_1) \wedge c_k$. 
	Suppose that $\mConv ({C}' \cup {D}')$ is join-closed.
	Then $(d_1 \vee u) \vee (d_1 \vee v) (= d_1 \vee (u \vee v))$ exists in 
	$[d_1, c_{k+1}] \cap \mConv (C' \cup D')$, and necessarily covers $u \vee v$.
	Thus $((d_1 \vee u) \vee (d_1 \vee v), c_k)$ is a modular pair with $((d_1 \vee u) \vee (d_1 \vee v)) \wedge c_k = u \vee v \in \Delta$.
	Suppose in addition that 
	$(u,v)$ is a modular pair. 
	Then $(u \vee d_1, v \vee d_1)$ is also modular.
	From $u \vee d_1, v \vee d_1 \in \mConv ({C}' \cup {D}')$, it holds 
	$(u \vee d_1) \wedge (v \vee d_1)$ belongs to  
	$[d_1,c_{k+1}] \cap \mConv ({C}' \cup {D}')$, and covers $u \wedge v$.
	This means that $((u \vee d_1) \wedge (v \vee d_1), c_k)$ 
	is a modular pair with	$(u \vee d_1) \wedge (v \vee d_1) \wedge c_k = u \wedge v$, implying $u \wedge v \in \Delta$.
\end{proof}

By using this lemma, we obtain 
a simple recursive algorithm  
to construct $\mConv ({C} \cup {D})$. 
This recursive algorithm can be naturally described 
by using the notion of the Jordan-H\"older permutation.
For $i \in [n]$, let $\sigma(i) = \sigma_{C,D}(i)$ denote 
the smallest index $j$ such that 
$d_{i} \leq d_{i-1} \vee c_j$.
The map $i \mapsto \sigma(i)$ 
is in fact a bijection on $[n]$, and 
is called the {\em Jordan-H\"older permutation}
with respect to ${C}, {D}$~\cite{Stanley72,Stanley74}. 
Observe that the above index $k$ 
is nothing but $\sigma(1) - 1$.

For $i=1,2,\ldots,n$, 
let $z_i, z_i'$ be elements defined by
\begin{equation}
	z_i' := c_{\sigma(i)} \vee d_{i} (= c_{\sigma(i)} \vee d_{i-1}),  \quad 
	z_{i} := c_{\sigma(i)-1} \vee d_{i-1}.
\end{equation}
Now $z_i$ is the maximal element of 
${C} \vee d_{i-1}$ not greater than $d_i$.
Therefore, the argument 
of the above lemma is applicable to two flags 
${C} \vee d_{i-1}, {D} \vee d_{i-1}$ in $[d_{i-1}, 1]$, and obtains 
$\mConv (({C} \vee d_{i-1}) \cup ({D} \vee d_{i-1}))$ from $\mConv (({C} \vee d_{i}) \cup ({D} \vee d_{i}))$.
Define $M_i \subseteq L$ ($i=n,n-1,\ldots,1$) by
\begin{eqnarray}
M_n  & := & \{ d_{n-1}, d_n \},  \\
M_{i} &:= & M_{i+1} \cup \{ q \wedge z_i \mid q \in [d_i, z'_i] \cap M_{i+1}: \mbox{$(q,z_i)$ is a modular pair} \}.
\end{eqnarray}

\begin{Cor}\label{cor:M1}
	$\mConv({C} \cup {D})$ is equal to $M_1$ and is join-closed.
\end{Cor}

Define $\varphi:\mConv({C} \cup {D}) \to 2^{[n]}$ by
\begin{equation}
\varphi(u) := \{ i \in [n] \mid z_i \geq u \}.
\end{equation}
\begin{Lem}\label{lem:varphi} The map
	$\varphi$ is an embedding 
	from $\mConv({C} \cup {D})$ to $2^{[n]}$ such that for $u,v \in \mConv({C} \cup {D})$,
	\begin{itemize}
		\item[{\rm (1)}] $r(u) = n - |\varphi(u)|$, and
		\item[{\rm (2)}] $\varphi(u \vee v) = \varphi(u) \cap \varphi(v)$.
	\end{itemize}
\end{Lem}
\begin{proof}
	Let $u \in \mConv({C} \cup {D})$.
	Since $M_1 = \mConv({C} \cup {D})$, 
	$u$ is the meet of a subset of $\{z_i\}$.
	Therefore $u = \bigwedge_{i \in \varphi(u)} z_i$, and $\varphi$ is injective. 
	Suppose that $u \in M_{i} \setminus M_{i+1}$.
	Then $u = q \wedge z_i$ for some $q \in [d_i,z_i'] \cap M_{i+1}$.
	It is clear that $r(u) = r(q) - 1$ 
	and $\varphi(u) = \varphi(q) \cup \{ z_i\}$ with $z_i \not \in \varphi(q)$. 
	By induction on $i$ from $n$, 
    we have (1), while (2) is obvious from the definition of $\varphi$.
\end{proof}

\begin{proof}[Proof of Theorem~\ref{thm:main}~(1)]
	Let $\bar \varphi: \mConv({C} \cup {D}) \to 2^{[n]}$ by $u \mapsto [n] \setminus \varphi(u)$.
	Then $\mConv({C} \cup {D})$ is isomorphic to the image of $\bar \varphi$, which is union-closed (A2) and 
	contains $\emptyset, [n]$ (A1).
	By Lemma~\ref{lem:varphi}~(1), 
	the image of ${C}$ (or $D$) is a flag required by (A3$'$).
	Therefore, the injective image is a pre-antimatroid.
\end{proof}

Next we show Theorem~\ref{thm:main}~(2).
We note that the flag-distance $\dist(C,D)$ is computed from Jordan-H\"older permutation $\sigma_{C,D}$.
For a permutation $\sigma$ on $[n]$, the {\em inversion number} of $\sigma$
is defined as the number of pairs 
$i,j \in [n]$ with $i<j$ and $\sigma(i) > \sigma(j)$.
\begin{Thm}[{\cite{AbelsA}; see \cite{Herscovici98}}]\label{thm:inversion}
	For flags $C,D$ in a semimodular lattice, 
	the flag-distance
	$\dist({C},{D})$ is equal to 
	the inversion number of $\sigma_{{C},{D}}$.
\end{Thm}
\begin{proof}[Sketch of Proof]
It is clear that the inversion number gives 
an upper bound of the flag-distance.
In the setting of the proof of Lemma~\ref{lem:1}, 
the inversion number of $\sigma_{{C},{D}}$
is $k$ plus the inversion number $\ell$ of $\sigma_{{C}',{D}'}$. 
By induction, ${C}' \cup \{0\}$ and ${D}$ 
are connected by a gallery of length $\ell$.
Also ${C'} \cup \{0\}$ and ${C}$
are connected by the gallery of length $k$, 
which consists of $k+1$ flags 
\[
0 = c_0 < c_1 < \cdots c_{i} < c_i \vee d_1 < \cdots < c_{k} \vee d_1 = c_{k+1} < \cdots < c_n = 1
\]
for $i=0,1,\ldots,k$.
\end{proof}

\begin{proof}[Proof of Theorem~\ref{thm:main}~(2)]
We first show that any flag in any shortest gallery belongs to $\mConv ({C} \cup {D})$.
Choose any flag ${D}'$ in a shortest gallery such that it is adjacent to ${D}$.
It suffices to show that ${D}' \subseteq \mConv ({C} \cup {D})$.
Here ${D}'$ is given by 
\[
0 = d_0 < d_1 < \cdots < d_{i-1} < d_{i}' < d_{i+1} < \cdots < d_{n} = 1
\]	
for $d'_i \neq d_i$.
By Theorem~\ref{thm:inversion}, it holds
$j:= \sigma_{{C}, {D}}(i) > \sigma_{{C}, {D}}(i+1) =:j'$, and $j= \sigma_{{C}, {D}'}(i+1)$.
Then $d_{i+1} \leq d_{i} \vee c_{j'} \leq d_{i} \vee c_{j}$, and $d_i \vee c_j = d_{i+1} \vee c_j$.
Both $d_{i-1}\vee c_{j-1}$ and $d_i' \vee c_{j-1}$ 
are covered by $d_{i+1} \vee c_j= d_i \vee c_j$. 
Necessarily, they are equal.
Now $d_{i+1}$ and $d_{i-1}\vee c_{j-1}$ form a modular pair such that their meet is $d'_i$.
Also $\mConv ({C}, {D})$ contains $d_{i-1}\vee c_{j-1}$ since it is join-closed (Corollary~\ref{cor:M1}). This concludes $d_i' \in \mConv ({C}, {D})$.
 
Next we show the conversion direction.
Regard $\mConv ({C} \cup {D})$ as a pre-antimatroid on $[n]$.
Then the cardinality $|X|$ for $X \in \mConv ({C} \cup {D})$ equals the rank $r(X)$ in $L$.
By this fact and the join-closedness, 
if $X,Y \in \mConv ({C} \cup {D})$  are moduler in $L$, 
then $X \cup Y, X \cap Y \in  \mConv ({C} \cup {D})$.
Therefore, $\mConv  ({C} \cup {D})$ is a subset of 
the sublattice generated by $C,D$ in the Boolean lattice $2^{[n]}$.
Abels~\cite[Proposition 1.8]{AbelsB} showed that a flag in 
this sublattice is precisely a flag in 
a shortest gallery between $C$ and $D$ in $2^{[n]}$.
On the other hand,  the antimatroid ${\cal A}:= (\mConv ({C} \cup {D}))^* \subseteq 2^{[n]}$ 
is a semimodular lattice 
with respect to the inclusion order, 
where the join and the rank are the same as in $L$ and as in $2^{[n]}$.
Consequently, the Jordan-H\"older permutation $\sigma_{{C},{D}}$ in ${\cal A}$
is the same as in $L$ and as in $2^{[n]}$.
Together with Theorem~\ref{thm:inversion}, 
we conclude that any flag in $\mConv({C}\cup {D})$ belongs to a shortest gallery in $L$.
\end{proof}

\section*{Acknowledgments}
The first author was supported by Grant-in-Aid for JSPS Research Fellow, Grant No. JP19J22605, Japan.
The second author was supported by JST PRESTO Grant Number JPMJPR192A, Japan.

\end{document}